\documentclass[10pt]{article}

\usepackage{amssymb,latexsym,amsmath,epsfig,amsthm,url} 

\newtheorem{theorem}{Theorem}
\newtheorem{lemma}{Lemma}

\newtheorem{corollary}{Corollary}

\theoremstyle{definition}

\newtheorem{remark}{Remark}

\begin{document}

\begin{center}
{\bf On the quantity $I(q^k) + I(n^2)$ where \\
$q^k n^2$ is an odd perfect number - Part II}
\vskip 20pt
{\bf Keneth Adrian Precillas Dagal}\\
{\tt Nasser Vocational Training Centre, Bahrain}\\
{\tt kendee2012@gmail.com}\\ 
\vskip 10pt
{\bf Jose Arnaldo Bebita Dris\footnote{Corresponding author}}\\
{\tt M.~Sc.~Graduate, De La Salle University, Manila, Philippines}\\
{\tt josearnaldobdris@gmail.com}\\ 
\end{center}
\vskip 20pt

\vskip 30pt

\centerline{\bf Abstract}
\noindent In this note, we continue an approach pursued in an earlier paper of the second author and thereby attempt to produce an improved lower bound for the sum $I(q^k) + I(n^2)$, where $q^k n^2$ is an odd perfect number with special prime $q$ and $I(x)$ is the abundancy index of the positive integer $x$.  In particular, this yields an upper bound for $k$.

\baselineskip=12.875pt
\vskip 30pt

\section{Introduction} 

Let $x$ be a positive integer.  We denote the sum of the divisors of $x$ by
$$\sigma(x) = \sum_{d \mid x}{d}.$$
We also denote the deficiency of $x$ by
$$D(x) = 2x - \sigma(x),$$
the sum of the aliquot or proper divisors of $x$ by
$$s(x) = \sigma(x) - x,$$
and the abundancy index of $x$ by $I(x)=\sigma(x)/x$.

Note that both $\sigma$ and $I$ are multiplicative. In general we have the inequalities
$$\sigma(yz) \leq \sigma(y)\sigma(z)$$
and
$$I(yz) \leq I(y)I(z)$$
for $\sigma$ and $I$.  Equality holds if and only if $\gcd(y,z)=1$.  Lastly, note that although the deficiency function $D$ is not multiplicative, it is in general true that the inequality
$$D(yz) \leq D(y)D(z)$$
holds whenever $\gcd(y,z)=1$.  This means that the deficiency function is submultiplicative.  In fact, the following formula relating the deficiency function and the sum-of-aliquot-divisors function holds for all $u$ and $v$ such that $\gcd(u,v)=1$:
$$D(u)D(v) - D(uv) = 2s(u)s(v).$$
(Note that this formula can be proven via a direct, ad hoc computation.)

If $m$ is odd and $\sigma(m)=2m$, then $m$ is called an odd perfect number.  Euler proved that an odd perfect number, if one exists, must have the form $m = q^k n^2$ where $q$ is the special prime satisfying $q \equiv k \equiv 1 \pmod 4$ and $\gcd(q,n)=1$.  Note that we have
$$\sigma(q^k)\sigma(n^2)=\sigma(q^k n^2)=\sigma(m)=2m=2 q^k n^2$$
so that we obtain
$$\frac{\sigma(n^2)}{q^k}=\frac{2n^2}{\sigma(q^k)}=\gcd(n^2,\sigma(n^2))=\frac{D(n^2)}{s(q^k)}=\frac{2s(n^2)}{D(q^k)}.$$

Descartes (1638) and Frenicle (1657) conjectured that $k=1$ always holds \cite{B}.  More recently, Sorli (2003) predicts $k=1$ after testing large numbers with eight distinct prime factors for perfection.  To date, no proof of the Descartes-Frenicle-Sorli Conjecture on odd perfect numbers is known, although various equivalent conditions have been derived by Dris (2017) \cite{D3}, and Dris and Tejada (2018) \cite{DT}.

In this note, we continue an approach pursued in an earlier paper by Dris (2020) \cite{D1}, building on previous work from the second author's masters thesis (2008) \cite{D5}, and thereby attempt to produce stronger bounds for the sum $I(q^k) + I(n^2)$.  Currently, we know by Dris (2012) \cite{D4} that
$$\frac{57}{20} < I(q^k) + I(n^2) < 3$$
and that these bounds are best-possible.

We also know that
$$\frac{q+1}{q} \leq I(q^k) < \frac{q}{q-1} < \frac{2(q-1)}{q} < I(n^2) \leq \frac{2q}{q+1}$$
from which we get
$$\bigg(I(q^k) - \frac{q}{q-1}\bigg)\bigg(I(n^2) - \frac{q}{q-1}\bigg) < 0$$
and
$$\bigg(I(q^k) - \frac{q+1}{q}\bigg)\bigg(I(n^2) - \frac{q+1}{q}\bigg) \geq 0.$$
Using the fact that $I(q^k)I(n^2)=I(q^k n^2)=2$, we obtain
$$\frac{2(q-1)}{q} + \frac{q}{q-1} < I(q^k) + I(n^2) \leq \frac{2q}{q+1} + \frac{q+1}{q}.$$
Notice that the lower bound equals
$$L(q) = \frac{2(q-1)}{q} + \frac{q}{q-1} = \frac{3q^2 - 4q + 2}{q(q-1)} = 3 - \bigg(\frac{q-2}{q(q-1)}\bigg)$$
and that the upper bound equals
$$U(q) = \frac{2q}{q+1} + \frac{q+1}{q} = \frac{3q^2 + 2q + 1}{q(q+1)} = 3 - \bigg(\frac{q-1}{q(q+1)}\bigg).$$
Equality holds in $L(q) < I(q^k) + I(n^2) \leq U(q)$ if and only if the Descartes-Frenicle-Sorli Conjecture on odd perfect numbers holds.

In the succeeding sections, we shall see how to successfully improve the lower bound for $I(q^k) + I(n^2)$.  This will translate to an upper bound for $k$.  In an e-mail sent to the second author in 2007, Iannucci asserted that an upper bound for $k$ ``would be a most helpful result, but is difficult to obtain".

\section{On $D(q^k)D(n^2)=2s(q^k)s(n^2)$ - Continued}

Since $\gcd(q^k,\sigma(q^k))=1$, we obtain the series of equations
$$\frac{\sigma(n^2)}{q^k}=\frac{2n^2}{\sigma(q^k)}=\gcd(n^2,\sigma(n^2))=\frac{D(n^2)}{s(q^k)}=\frac{2s(n^2)}{D(q^k)}$$
from which we obtain
$$D(q^k)D(n^2)=2s(q^k)s(n^2),$$
which was the primary object of investigation in the paper \cite{D1}.

We recall the following lemma from \cite{D1}, and hereby produce a shorter proof for the same.

\begin{lemma}\label{CommonValue}
If $m = q^k n^2$ is an odd perfect number with special prime $q$, then
$$D(q^k)D(n^2)=2s(q^k)s(n^2)=\frac{2n^2(q^k-1)(q^{k+1}-2q^k+1)}{(q-1)(q^{k+1}-1)}.$$ 
\end{lemma}

\begin{proof} 
Let $m= q^k n^2$ be an odd perfect number with special prime $q$.

Since $q$ is prime,
$$D(q^k)=2q^k-\sigma(q^k)=2q^k-\bigg(\frac{q^{k+1}-1}{q-1}\bigg)=\frac{2q^{k+1}-2q^k-q^{k+1}+1}{q-1}=\frac{q^{k+1}-2q^k+1}{q-1}.$$
Additionally, from the series of equations above, we obtain
$$D(n^2)=\frac{2n^2 s(q^k)}{\sigma(q^k)}=\frac{2n^2 (q^k - 1)}{q^{k+1} - 1}.$$

Finally, we get
$$D(q^k)D(n^2)=\bigg(\frac{q^{k+1}-2q^k+1}{q-1}\bigg)\cdot\bigg(\frac{2n^2 (q^k - 1)}{q^{k+1} - 1}\bigg)$$
and the result follows by commutativity, and from the equation
$$D(u)D(v) - D(uv) = 2s(u)s(v),$$
where we set $u=q^k$ and $v=n^2$, noting that $\gcd(u,v)=\gcd(q^k, n^2)=\gcd(q,n)=1$ and that $D(q^k n^2)=0$ since $m = q^k n^2$ is perfect.
\end{proof}

We restate the following theorem from the paper \cite{D1}, without proof, as it can be directly proved using Lemma \ref{CommonValue}.

\begin{theorem}\label{MainResult1}
If $m = q^k n^2$ is an odd perfect number with special prime $q$, then
$$3 - (I(q^k) + I(n^2)) = \frac{(q^k-1)(q^{k+1}-2q^k+1)}{{q^k}(q-1)(q^{k+1}-1)}.$$
\end{theorem}

\section{Attempting to improve the bounds \\ 
for $I(q^k) + I(n^2)$}

Since $q$ is prime and $m = q^k n^2$ is perfect, we obtain
$$\frac{q+1}{q} \leq I(q^k) < \frac{q}{q-1}$$
so that we get
$$\frac{2(q-1)}{q} < I(n^2) = \frac{2}{I(q^k)} \leq \frac{2q}{q+1}.$$
Using the identity
$$\frac{D(x)}{x} = 2 - I(x)$$
we obtain the bounds
$${q^k}\bigg(\frac{q-2}{q-1}\bigg) < D(q^k) \leq {q^{k-1}}\bigg(q-1\bigg)$$
and
$$\frac{2n^2}{q+1} \leq D(n^2) < \frac{2n^2}{q}.$$
This implies that
$${2q^k n^2}\cdot\bigg(\frac{q-2}{(q-1)(q+1)}\bigg) < D(q^k)D(n^2) < {2q^k n^2}\cdot\bigg(\frac{q-1}{q^2}\bigg).$$
Dividing both sides of the last inequality by $2q^k n^2$, we get
$$\frac{q-2}{(q-1)(q+1)} < \frac{D(q^k)D(n^2)}{2q^k n^2} < \frac{q-1}{q^2}.$$
Thus, the fraction in the middle of the last inequality simplifies to
$$\frac{D(q^k)D(n^2)}{2q^k n^2}=\bigg(\frac{s(q^k)}{q^k}\bigg)\bigg(\frac{s(n^2)}{n^2}\bigg)=\bigg(I(q^k) - 1\bigg)\bigg(I(n^2) - 1\bigg) = 3 - \bigg(I(q^k) + I(n^2)\bigg).$$

We therefore finally have the bounds
$$3 - \bigg(\frac{q-1}{q^2}\bigg) < I(q^k) + I(n^2) < 3 - \bigg(\frac{q-2}{(q-1)(q+1)}\bigg)$$
which does not improve on the known bounds
$$3 - \bigg(\frac{q-2}{q(q-1)}\bigg) < I(q^k) + I(n^2) \leq 3 - \bigg(\frac{q-1}{q(q+1)}\bigg).$$

While this previous attempt is unsuccessful, let us see what we could obtain from the following equation:

\begin{lemma}\label{GCD}
If $q^k n^2$ is an odd perfect number with special prime $q$, then
$$\gcd(n^2, \sigma(n^2)) = q\sigma(n^2) - 2(q - 1)n^2.$$
\end{lemma}

\begin{proof}
From the Introduction, we have the equation
$$\frac{\sigma(n^2)}{q^k}=\frac{D(n^2)}{s(q^k)}=\gcd(n^2, \sigma(n^2)).$$
We use the identity
$$\frac{A}{B}=\frac{C}{D}=\frac{A - C}{B - D}$$
which will work if we set $A = \sigma(n^2)$, $B = q^k$, $C = (q - 1)D(n^2)$, and $D = q^k - 1$.  Note that $B - D \neq 0$.

We finally obtain
$$\gcd(n^2, \sigma(n^2)) = \sigma(n^2) - (q - 1)(2n^2 - \sigma(n^2)) = \sigma(n^2) - 2qn^2 + 2n^2 + q\sigma(n^2) - \sigma(n^2)$$ 
$$= q\sigma(n^2) - 2(q - 1)n^2.$$
\end{proof}

\begin{remark}\label{CrucialObservation}
In particular, we get that
$$I(n^2) - \frac{2(q - 1)}{q} = \frac{I(n^2)}{q^{k+1}} = \frac{\sigma(n^2)}{q^k}\cdot\frac{1}{qn^2}.$$
\end{remark}

From the crucial observation in Remark \ref{CrucialObservation}, we get the following improved lower bound for $I(n^2)$:

\begin{lemma}\label{ImprovedLowerBoundforI(n^2)}
If $q^k n^2$ is an odd perfect number with special prime $q$, then
$$I(n^2) > \frac{2(q - 1)}{q} + \frac{1}{qn^2}.$$
\end{lemma}

\begin{proof}
The proof follows by considering the lower bound $\sigma(n^2)/q^k > 1$ from \cite{D4}.
\end{proof}

We shall use Lemma \ref{ImprovedLowerBoundforI(n^2)} to get an upper bound for $k$ in the succeeding sections.

We can generalize Lemma \ref{ImprovedLowerBoundforI(n^2)} to the following corollary, if we consider a generic (and strict) lower bound $\rho$ for $\sigma(n^2)/q^k$:

\begin{corollary}\label{ImprovedLowerBoundforI(n^2)RHO}
If $q^k n^2$ is an odd perfect number with special prime $q$, then
$$I(n^2) > \frac{2(q - 1)}{q} + \frac{\rho}{qn^2}.$$
\end{corollary}

Note that
$$I(q^k)+I(n^2)=\frac{3q^{2k+2} - 4q^{2k+1} + 2q^{2k} - 2q^{k+1} + 1}{q^k (q - 1)(q^{k+1} - 1)}$$
and that WolframAlpha gives the partial fraction decomposition
$$I(q^k)+I(n^2)=\frac{3q^2-4q+2}{q(q-1)}+\frac{2(q-1)}{q(q^{k+1}-1)}-\frac{1}{q^k(q-1)}.$$

\subsection{Improving the upper bound $I(q^k)+I(n^2) \leq 3 - \bigg(\frac{q-1}{q(q+1)}\bigg)$}\label{UpperBound}

By Theorem \ref{MainResult1}, let 
$$f(k) = 3 - (I(q^k) + I(n^2)) = \frac{(q^k-1)(q^{k+1}-2q^k+1)}{{q^k}(q-1)(q^{k+1}-1)}.$$

Letting
$$g(k) = I(q^k) + I(n^2),$$
it follows that
$$g'(k) = -f'(k).$$
But we know from the results of the paper \cite{D1} that
$$f'(k) = \bigg(\frac{ (q-4)q^{2 k + 1} +2 q^{k + 1}  + 2 q^{2 k} - 1}{q^k(q - 1) (q^{k + 1} - 1)^2}\bigg)\ln q,$$
which is positive for $k \geq 1$ and $q \geq 5$.
Thus, $g'(k) < 0$ for all $k$.

It follows that
$$\lim_{k \rightarrow \infty}{g(k)} < g(k) \leq g(1).$$

But
$$\lim_{k \rightarrow \infty}{g(k)} = 3 - \bigg(\frac{q-2}{q(q-1)}\bigg)$$
and
$$g(1) = 3 - \bigg(\frac{q-1}{q(q+1)}\bigg),$$
which are just the known bounds.

We can now prove the following theorem.

\begin{theorem}\label{MainResult2}
Let $m = q^k n^2$ be an odd perfect number with special prime $q$.
\begin{enumerate}
\item{If $k \neq 1$ can be proved, then one can obtain an improved upper bound for $I(q^k) + I(n^2)$.}
\item{If one can obtain an improved upper bound for $I(q^k) + I(n^2)$, then $k \neq 1$ can be proved.}
\end{enumerate}
\end{theorem}

\begin{proof}
For the first implication, we just need to use the fact that $g(k)=I(q^k)+I(n^2)$ is decreasing, so that if $k \neq 1$, then $k \geq 5$, which implies that $g(k) \leq g(5)$, where of course we have $g(5) < g(1)$ (since $g(k)$ is decreasing).

For the second implication, assume that an improved upper bound of $h_1(q)<g(1)$ is obtained for $I(q^k)+I(n^2)$.  Assume to the contrary that $k=1$.  Then
$$g(1)=I(q^k)+I(n^2) \leq h_1(q) < g(1),$$
which is a contradiction.  It follows that $k \neq 1$.

This finishes the proof. 
\end{proof}

\subsection{Improving the lower bound $3 - \bigg(\frac{q-2}{q(q-1)}\bigg) < I(q^k)+I(n^2)$}\label{LowerBound}

First, note that trivially, we have the upper bound $k < m = q^k n^2$.  Similar to the approach in Subsection \ref{UpperBound}, we can prove the following theorem.

\begin{theorem}\label{MainResult3}
Let $m = q^k n^2$ be an odd perfect number with special prime $q$.
\begin{enumerate}
\item{If it can be proven that there is an integer $e$ such that $k \leq e$, then one can obtain an improved lower bound for $I(q^k) + I(n^2)$.}
\item{If one can obtain an improved lower bound for $I(q^k) + I(n^2)$, then it can be proven that there is an integer $e$ such that $k \leq e$.}
\end{enumerate}
\end{theorem}

\begin{proof}
Suppose that $m = q^k n^2$ is a hypothetical odd perfect number with special prime $q$.

Assume that there is an integer $e$ such that $k \leq e$.  As before, let
$$g(k) = I(q^k) + I(n^2).$$
By the considerations in Subsection \ref{UpperBound}, $g'(k) < 0$ for all $k$, so that the function $g(k)$ is decreasing.  This means that, considering the graph of $Y = g(k)$, we see that $g(e)$ is an improved lower bound since $k \leq e$ implies $g(e) \leq g(k)$.

For the second implication, if one gets an improved lower bound $h_2(q)$, then from the proof of the first implication, since $g(k)$ is (strictly) decreasing, $g(k)$ is one-to-one.  Therefore, there is only one $k = k_0$ such that $g(k) = h(q)$.  Then we have
$$k \leq \lceil{k_0}\rceil.$$

This concludes the proof.
\end{proof}

Now, let us see what we have got from Lemma \ref{ImprovedLowerBoundforI(n^2)}:

\begin{theorem}\label{ImprovedLowerBoundforI(q^k)+I(n^2)}
If $q^k n^2$ is an odd perfect number with special prime $q$, then
$$I(q^k) + I(n^2) > \dfrac{3q^2 - 4q + 2}{q(q - 1)} - \dfrac{q}{(q - 1)(2qn^2 - 2n^2 + 1)} + \dfrac{1}{qn^2} = l_1(q,n).$$
In particular, this means that there exists a number $K$ such that $k < K$ and \linebreak $g(K) = l_1(q,n)$.
\end{theorem}

\begin{proof}
From Lemma \ref{ImprovedLowerBoundforI(n^2)}, we have
$$I(n^2) > \frac{2(q - 1)}{q}+\frac{1}{qn^2}.$$
But we know that
$$I(q^k) < \frac{q}{q - 1} < \frac{2(q - 1)}{q}.$$
Hence, we obtain
$$I(q^k) < \frac{2(q - 1)}{q}+\frac{1}{qn^2} < I(n^2).$$
Consider the product
$$\Bigg(I(q^k) - \bigg(\frac{2(q - 1)}{q}+\frac{1}{qn^2}\bigg)\Bigg)\Bigg(I(n^2) - \bigg(\frac{2(q - 1)}{q}+\frac{1}{qn^2}\bigg)\Bigg).$$
This product is negative.  Consequently, we get
$$2+\bigg(\frac{2(q - 1)}{q}+\frac{1}{qn^2}\bigg)^2 < \bigg(\frac{2(q - 1)}{q}+\frac{1}{qn^2}\bigg)\bigg(I(q^k) + I(n^2)\bigg),$$
from which we finally obtain
$$I(q^k) + I(n^2) > \frac{2}{\bigg(\frac{2(q - 1)}{q}+\frac{1}{qn^2}\bigg)} + \bigg(\frac{2(q - 1)}{q}+\frac{1}{qn^2}\bigg) = l_1(q,n).$$
WolframAlpha computes the partial fraction decomposition of $l_1(q,n)$ as
$$l_1(q,n) = \frac{3q^2 - 4q + 2}{q(q - 1)} - \frac{q}{(q - 1)(2qn^2 - 2n^2 + 1)} + \frac{1}{qn^2}.$$
Note that 
$$-\dfrac{q}{(q - 1)(2qn^2 - 2n^2 + 1)} + \dfrac{1}{qn^2} = \dfrac{qn^2 (q - 4) + q + 2n^2 - 1}{qn^2 (q - 1)(2n^2 (q - 1) + 1)} > 0$$
holds, since $q$ is a prime satisfying $q \equiv 1 \pmod 4$ implies that $q \geq 5$.  Hence,
$$l_1(q,n) = \frac{3q^2 - 4q + 2}{q(q - 1)} - \frac{q}{(q - 1)(2qn^2 - 2n^2 + 1)} + \frac{1}{qn^2}$$
is an improved lower bound for $I(q^k) + I(n^2)$, better than
$$I(q^k) + I(n^2) > \frac{3q^2 - 4q + 2}{q(q - 1)}.$$
Note that $l_1(q,n)$ does not contain $k$.

Lastly, we need to check that, indeed
$$g(1) - l_1(q,n) = \bigg(\frac{3q^2 + 2q + 1}{q(q + 1)}\bigg) - \bigg(\frac{3q^2 - 4q + 2}{q(q - 1)} - \frac{q}{(q - 1)(2qn^2 - 2n^2 + 1)} + \dfrac{1}{qn^2}\bigg)$$
is positive.  We compute
$$g(1) - l_1(q,n) = \frac{\bigg((q-3)n^2 + 1\bigg)\bigg(2n^2 - q - 1\bigg)}{n^2 q(q + 1) \bigg(2n^2 (q - 1) + 1\bigg)}$$
which is indeed positive since 
$$5 \leq q < 2n^2 - 1$$
follows from $q$ being a prime satisfying $q \equiv 1 \pmod 4$, and the result
$$q + 1 \leq \sigma(q^k) \leq \frac{2n^2}{3} < 2n^2$$
from Dris \cite{D4}. 

By Theorem \ref{MainResult3}, we infer that there exists a number $K$ such that $1 \leq k < K$ and $g(K) = l_1(q,n)$.
\end{proof}

We now compute an explicit upper bound for $K$, in terms of $q$ and $n$.

\begin{theorem}\label{ExplicitUpperBoundForKInTermsofqAndn}
If $g(K) = l_1(q,n)$, then 
$$\frac{q^{K+1} - 1}{q - 1} = 2n^2.$$
In particular, 
$$K < \log_q{2} + 2\log_q{n}.$$
\end{theorem}

\begin{proof}
We require $g(K) = l_1(q,n)$.

But we have the partial fraction decompositions
$$g(K) = \frac{3q^2 - 4q + 2}{q(q - 1)} + \frac{2(q - 1)}{q(q^{K+1} - 1)} - \frac{1}{{q^K}(q - 1)}$$
and
$$l_1(q,n) = \frac{3q^2 - 4q + 2}{q(q - 1)} - \frac{q}{(q - 1)(2qn^2 - 2n^2 + 1)} + \frac{1}{qn^2},$$
as computed by WolframAlpha.

Rearranging terms from the equation $g(K) = l_1(q,n)$ then gives
$$\frac{2(q - 1)}{q(q^{K+1} - 1)} - \frac{1}{qn^2} = \frac{1}{{q^K}(q - 1)} - \frac{q}{(q - 1)(2qn^2 - 2n^2 + 1)}.$$

After some algebraic simplifications, we obtain
$$\frac{2n^2 (q - 1) - (q^{K+1} - 1)}{qn^2 (q^{K+1} - 1)} = \frac{2qn^2 - 2n^2 + 1 - q^{K+1}}{{q^K} (q - 1)(2qn^2 - 2n^2 + 1)},$$
from which we finally get
$$\frac{2qn^2 - 2n^2 + 1 - q^{K+1}}{qn^2 (q^{K+1} - 1)} = \frac{2qn^2 - 2n^2 + 1 - q^{K+1}}{{q^K} (q - 1)(2qn^2 - 2n^2 + 1)}.$$

Suppose to the contrary that
$$2qn^2 - 2n^2 + 1 - q^{K+1} \neq 0.$$
Then we may cancel $2qn^2 - 2n^2 + 1 - q^{K+1}$ in the numerator of both sides of the equation, to get
$${q^K} (q - 1)(2qn^2 - 2n^2 + 1) = qn^2 (q^{K+1} - 1).$$
This may be rewritten as
$$2 - \frac{2}{q} + \frac{1}{qn^2} = \frac{2qn^2 - 2n^2 + 1}{qn^2} = \frac{q^{K+1} - 1}{q^K (q - 1)}.$$
But we know of the estimates
$$\frac{2(q - 1)}{q} < 2 - \frac{2}{q} + \frac{1}{qn^2} = \frac{q^{K+1} - 1}{q^K (q - 1)} < \frac{q}{q - 1}.$$
These estimates imply that
$$\sqrt{2} < \frac{q}{q - 1},$$
contradicting
$$\frac{q}{q - 1} \leq \frac{5}{4},$$
since $q$ is the special prime satisfying $q \equiv 1 \pmod 4$ implies that $q \geq 5$.

The contradiction thus obtained means that our assumption that
$$2qn^2 - 2n^2 + 1 - q^{K+1} \neq 0$$
is untenable.  This implies that
$$2qn^2 - 2n^2 + 1 - q^{K+1} = 0,$$
from which we obtain
$$2n^2 (q - 1) = q^{K+1} - 1$$
$$2n^2 = \frac{q^{K+1} - 1}{q - 1}.$$

Note that $K$ may not be an integer.  We wish to show that
$$q^K < \frac{q^{K+1} - 1}{q - 1}.$$
Suppose to the contrary that
$$\frac{q^{K+1} - 1}{q - 1} \leq q^K.$$
We then get
$$q^{K+1} - 1 \leq q^{K+1} - q^K$$
$$q^K \leq 1$$
which contradicts
$$1 \leq k < K$$
and $q \geq 5$.

Hence, we obtain
$$q^K < 2n^2$$
$$K \log{q} < \log{2} + 2\log{n}$$
Finally, we get the upper bound
$$K < \log_q{2} + 2\log_q{n}.$$
\end{proof}

We can generalize Theorem \ref{ImprovedLowerBoundforI(q^k)+I(n^2)} and Theorem \ref{ExplicitUpperBoundForKInTermsofqAndn} in the following corollary, by using Corollary \ref{ImprovedLowerBoundforI(n^2)RHO}:

\begin{corollary}\label{ImprovedLowerBoundforI(q^k)+I(n^2)RHO}
If $q^k n^2$ is an odd perfect number with special prime $q$, then
$$I(q^k) + I(n^2) > \dfrac{3q^2 - 4q + 2}{q(q - 1)} - \dfrac{\rho q}{(q - 1)(2qn^2 - 2n^2 + \rho)} + \dfrac{\rho}{qn^2} = l_{\rho}(q,n).$$
In particular, this means that there exists a number $K'$ such that $k < K'$ and \linebreak $g(K') = l_{\rho}(q,n)$.  We thereby compute that
$$K' < \log_q {2} + 2\log_q {n} - \log_q {\rho}.$$
\end{corollary}

\begin{proof}
Abbreviate a (strict) lower bound for the quantity
$$\frac{\sigma(n^2)}{q^k}=\frac{2n^2}{\sigma(q^k)}=\gcd(n^2,\sigma(n^2))$$
by $\rho$.

By Corollary \ref{ImprovedLowerBoundforI(n^2)RHO}, we then get the bound
$$I(q^k) < \dfrac{q}{q-1} < \dfrac{2(q-1)}{q} + \dfrac{\rho}{qn^2} < I(n^2)$$
which implies that the product
$$(I(q^k) - y)(I(n^2) - y) < 0$$
is negative, where $$y = \dfrac{2(q-1)}{q} + \dfrac{\rho}{qn^2}.$$  After some careful algebraic simplifications, we get
$$I(q^k) + I(n^2) > \frac{2qn^2}{2qn^2 - 2n^2 + \rho} + \frac{2qn^2 - 2n^2 + \rho}{qn^2}$$
which has the partial fraction decomposition
$$\frac{2qn^2}{2qn^2 - 2n^2 + \rho} + \frac{2qn^2 - 2n^2 + \rho}{qn^2} = \frac{3q^2 - 4q + 2}{q(q - 1)} - \frac{\rho q}{(q - 1)(2qn^2 - 2n^2 + \rho)} + \frac{\rho}{qn^2}.$$
Therefore,
$$I(q^k) + I(n^2) > \dfrac{3q^2 - 4q + 2}{q(q - 1)} - \dfrac{\rho q}{(q - 1)(2qn^2 - 2n^2 + \rho)} + \dfrac{\rho}{qn^2}.$$

Since it is known that $\rho > 1$ holds, and that $q \geq 5$, then we also know that
$$-\frac{\rho q}{(q - 1)(2qn^2 - 2n^2 + \rho)} + \frac{\rho}{qn^2} = \frac{\rho \Bigg(qn^2 (q - 4) + \rho(q - 1) + 2n^2\Bigg)}{qn^2 (q - 1)(2n^2 (q - 1) + \rho)} > 0.$$

This means that the new lower bound
$$I(q^k) + I(n^2) > l_{\rho}(q,n)$$
where
$$l_{\rho}(q,n) = \frac{2qn^2}{2qn^2 - 2n^2 + \rho} + \frac{2qn^2 - 2n^2 + \rho}{qn^2}$$
improves on the old (and trivial) lower bound
$$I(q^k) + I(n^2) > \frac{3q^2 - 4q + 2}{q(q - 1)}.$$
Note that $l_{\rho}(q,n)$ does not contain $k$.

Lastly, we need to check that, indeed
$$g(1) - l_{\rho}(q,n) = \bigg(\frac{3q^2 + 2q + 1}{q(q + 1)}\bigg) - \bigg(\frac{3q^2 - 4q + 2}{q(q - 1)} - \frac{\rho q}{(q - 1)(2qn^2 - 2n^2 + \rho)} + \dfrac{\rho}{qn^2}\bigg)$$
is positive.  We compute
$$g(1) - l_{\rho}(q,n) = \frac{\bigg((q-3)n^2 + \rho\bigg)\bigg(2n^2 - \rho(q + 1)\bigg)}{n^2 q(q + 1) \bigg(2n^2 (q - 1) + \rho\bigg)}$$
which is indeed positive since 
$$\rho < \frac{\sigma(n^2)}{q^k} = \frac{2n^2}{\sigma(q^k)} \leq \frac{2n^2}{q + 1}$$
by assumption.

By Theorem \ref{MainResult3}, we know that there exists a number $K'$ such that $1 \leq k < K'$.

We now compute an explicit upper bound for $K'$, in terms of $q$, $n$, and $\rho$.

So here we go:  We require $g(K') = l_{\rho}(q,n)$.

But we have the partial fraction decompositions
$$g(K') = \frac{3q^2 - 4q + 2}{q(q - 1)} + \frac{2(q - 1)}{q(q^{K' + 1} - 1)} - \frac{1}{q^{K'}(q - 1)}$$
and
$$l_{\rho}(q,n) = \frac{3q^2 - 4q + 2}{q(q - 1)} - \frac{\rho q}{(q - 1)(2qn^2 - 2n^2 + \rho)} + \frac{\rho}{qn^2}.$$

Equating and rearranging as before, we obtain
$$\frac{2(q - 1)}{q(q^{K' + 1} - 1)} - \frac{\rho}{qn^2} = \frac{1}{q^{K'}(q - 1)} - \frac{\rho q}{(q - 1)(2qn^2 - 2n^2 + \rho)}.$$

After some algebraic simplifications, we get
$$\frac{2n^2 (q - 1) - \rho\bigg(q^{K' + 1} - 1\bigg)}{qn^2 \bigg(q^{K' + 1} - 1\bigg)} = \frac{2n^2 (q - 1) + \rho - \rho q^{K' + 1}}{{q^{K'}}(q - 1)(2qn^2 - 2n^2 + \rho)}.$$

Proceeding similarly as in Theorem \ref{ImprovedLowerBoundforI(q^k)+I(n^2)}, suppose to the contrary that
$$2n^2 (q - 1) - \rho\bigg(q^{K' + 1} - 1\bigg) \neq 0.$$

Then we can cancel the numerator of both sides of the equation, since
$$2n^2 (q - 1) - \rho\bigg(q^{K' + 1} - 1\bigg) = 2n^2 (q - 1) + \rho - \rho q^{K' + 1}.$$

We thus obtain
$$qn^2 \bigg(q^{K' + 1} - 1\bigg) = {q^{K'}}(q - 1)(2qn^2 - 2n^2 + \rho)$$
which can be rewritten as
$$\frac{q^{K' + 1} - 1}{{q^{K'}}(q - 1)} = \frac{2qn^2 - 2n^2 + \rho}{qn^2}.$$

But, as before, we have the estimates
$$\frac{2(q - 1)}{q} < \frac{2qn^2 - 2n^2 + \rho}{qn^2}$$
(since $\rho$ is positive), and
$$\frac{q^{K' + 1} - 1}{{q^{K'}}(q - 1)} < \frac{q}{q - 1},$$
which (again) implies that
$$\sqrt{2} < \frac{q}{q - 1},$$
contradicting
$$\frac{q}{q - 1} \leq \frac{5}{4}$$
since $q$ is the special prime satisfying $q \equiv 1 \pmod 4$ implies that $q \geq 5$.

The contradiction thus obtained means that our assumption that
$$2n^2 (q - 1) - \rho\bigg(q^{K' + 1} - 1\bigg) \neq 0$$
is untenable.  This implies that
$$\frac{2n^2}{\rho} = \frac{q^{K' + 1} - 1}{q - 1}.$$

Note that $K'$ may not be an integer.

Proceeding similarly as in Theorem \ref{ImprovedLowerBoundforI(q^k)+I(n^2)}, we have
$$q^{K'} < \frac{q^{K' + 1} - 1}{q - 1}$$
which implies that
$$q^{K'} < \frac{2n^2}{\rho}.$$

We finally obtain the upper bound
$$K' < \log_q {2} + 2\log_q {n} - \log_q {\rho}.$$

\end{proof}

\begin{remark}
Per \url{https://math.stackexchange.com/a/4028814/28816}, the best currently known lower bound for $\sigma(n^2)/q^k$ is
$$\frac{\sigma(n^2)}{q^k} \geq {3^3} \times {5^3} = 3375.$$
\end{remark}

\section{Concluding Remarks and Further Research}

It may be possible to compute for an approximate value for $K$ or $K'$ from the equations 
$$\frac{q^{K+1} - 1}{q - 1}=2n^2$$
and
$$\frac{q^{K'+1} - 1}{q - 1}=\frac{2n^2}{\rho}$$
using numerical methods.

We leave these considerations and the problems they spur for the resolution of other researchers.

\vskip 20pt
\noindent {\bf Acknowledgements}
\vskip 20pt

The authors are indebted to the anonymous referees whose valuable feedback helped in improving the quality of this manuscript. \linebreak The second author is also grateful to the anonymous MSE user mathlove \linebreak (\url{https://math.stackexchange.com/users/78967}) for patiently answering his inquiries \cite{MD} in this Q\&A site.  The second author would also like to dedicate the proof of Theorem \ref{ImprovedLowerBoundforI(q^k)+I(n^2)} to Doug Iannucci, who introduced him to the said problem in the year 2007, while he was in the process of writing his M.~Sc. thesis.


\begin{thebibliography}{10}\footnotesize

\bibitem{B} B. Beasley, Euler and the ongoing search for odd perfect numbers, {\it ACMS 19th Biennial Conference Proceedings}, Bethel University (2013).
\bibitem{DT} J. Dris and D. Tejada, Conditions equivalent to the Descartes-Frenicle-Sorli Conjecture on odd perfect numbers - Part II, {\it Notes Number Theory Discrete Math}, {\bf 24 (3)} (2018), 62--67.
\bibitem{D1} J. Dris, On the quantity $I(q^k) + I(n^2)$ where $q^k n^2$ is an odd perfect number, {\it Notes Number Theory Discrete Math}, {\bf 26 (3)} (2020), 25--32.
\bibitem{D2} J. Dris, The abundancy index of divisors of odd perfect numbers – Part III, {\it Notes Number Theory Discrete Math}, {\bf 23 (3)} (2017), 53--59.
\bibitem{D3} J. Dris, Conditions equivalent to the Descartes-Frenicle-Sorli Conjecture on odd perfect numbers, {\it Notes Number Theory Discrete Math}, {\bf 23 (2)} (2017), 12--20.
\bibitem{D4} J. Dris, The abundancy index of divisors of odd perfect numbers, {\it J. Integ. Seq.} {\bf 15 (4)} (2012), Article 12.4.4.
\bibitem{D5} J. Dris, {\it Solving the Odd Perfect Number Problem: Some Old and New Approaches}, 
M.~Sc.~thesis, De La Salle University, Manila, Philippines, 2008.
\bibitem{MD} MSE user mathlove and J. Dris, Does the following lower bound improve on $I(q^k) + I(n^2) > 3 - \frac{q-2}{q(q-1)}$, where $q^k n^2$ is an odd perfect number? - Part II, \url{https://math.stackexchange.com/questions/4197182}, Last updated on Aug. 3, 2021.

\end{thebibliography}
\end{document}